\documentclass{article}
\usepackage{amsmath,amsthm,amsfonts,graphicx}
\usepackage{multirow}
\usepackage{slashbox}
\usepackage{multirow}
\def\smallddots{\mathinner{\raise7pt\hbox{.}\raise4pt\hbox{.}\raise1pt\hbox{.}}}
\def\smallsdots{\mathinner{\raise1pt\hbox{.}\raise4pt\hbox{.}\raise7pt\hbox{.}}}

\DeclareMathOperator{\diag}{diag}

\numberwithin{equation}{section}
\numberwithin{table}{section}
\newtheorem{theorem}{Theorem}[section]

\newtheorem{algorithm}{Algorithm}[section]

\newtheorem{definition}{Definition}[section]

\newtheorem{remark}{Remark}[section]

\begin{document}

\title{\bf Novel Approach to Real Polynomial Root-finding and Matrix Eigen-solving}
\author{Victor Y. Pan  \\
 Department of Mathematics and Computer Science \\
Lehman College of the City University of New York \\
Bronx, NY 10468 USA \\
 Ph.D. Programs in Mathematics  and Computer Science \\
The Graduate Center of the City University of New York \\
New York, NY 10036 USA \\
 victor.pan@lehman.cuny.edu \\
http://comet.lehman.cuny.edu/vpan/  \\
(This work is supported by NSF Grant CCF 1116736.  \\
Some of its results are to be presented at CASC 2014.)}
 
 \date{}

\maketitle


\maketitle


\begin{abstract}
Univariate polynomial root-finding is both classical and
important for modern computing. Frequently one seeks
just the real roots of a polynomial with real coefficients.
They can be approximated at a low computational cost if the
polynomial has no nonreal roots, but typically nonreal roots are
much more numerous than the real ones. We dramatically accelerate
the known algorithms in this case by exploiting the correlation
between the computations with matrices and polynomials, 
extending the techniques of the matrix sign iteration, and 
 exploiting the structure of the companion matrix of
the input polynomial. We extend some of the proposed techniques 
to the approximation of the real eigenvalues of 
a real nonsymmetric matrix. 
  \end{abstract}



\section{Introduction}
Assume
a univariate polynomial of degree $n$  with
 real coefficients,
\begin{equation}\label{eqpoly}
 p(x)=\sum^{n}_{i=0}p_ix^i=p_n\prod^n_{j=1}(x-x_j),~~~ p_n\ne 0,
\end{equation}
which  has $r$ real roots 
  $x_1,\dots,x_r$ and $s=(n-r)/2$
pairs of nonreal complex conjugate roots.
In some applications, e.g., to algebraic and geometric optimization,
one seeks just the $r$ real roots, which 
make up just a small fraction of all roots. 
 This is a well studied subject (see
 \cite[Chapter 15]{MP13}, \cite{PT13}, \cite{SMa},
and the  bibliography therein), but we 
dramatically accelerate the known algorithms 
by combining and extending the techniques
of \cite{PZ11} and \cite{PQZ12}. At first our iterative Algorithm
\ref{alg1}
reduces the original problem of real root-finding
to the same problem for an auxiliary
polynomial of degree $r$ having  $r$ real roots. 
Our iterations converge with quadratic 
rate, and so we need only $k=O(b+d)$ iterations, assuming  the tolerance $2^{-b}$ to the
error norm of the approximation to the auxiliary polynomial (we denote it $v_k(x)$) and the minimal distance
$2^{-d}$ of the nonreal roots from the real axis. The
values $d$ and $k$ are large for the input polynomials with nonreal
roots lying very close to the real axis, but our 
 techniques of Remark \ref{renrlr} enable us to
handle such harder inputs as well.
The known algorithms approximate the  roots of
 $v_k(x)$ 
 at a low arithmetic cost, and 
having these approximations computed, we recover the $r$
 real roots of the input polynomial $p(x)$.
Overall we perform $O(kn\log (n))$
 arithmetic operations. 
This arithmetic cost bound  is quite low, 
but in the case of large degree $n$, the algorithm is prone to
numerical problems, and so we devise dual
Algorithms
\ref{alg2} and \ref{alg2a} to avoid the latter problems.
This works quite well according to our test results, 
but formal
study of the issue and of the Boolean complexity of the
algorithm is left as a research challenge.

Let us comment briefly on the techniques involved and the complexity of the 
latter algorithms.
They perform computations in the Frobenius
matrix algebra generated by the companion matrix of the input
polynomial. By using FFT  and exploiting the structure of the matrices in the
algebra, one can operate with them as fast as with
polynomials. Real polynomial root-finding is
reduced to real eigen-solving for the companion
matrix.  
Transition to matrices and 
the randomization techniques, extended from  \cite[Section
5]{PQZ12},
streamline and simplify
the iterative process of Algorithm
\ref{alg1}. Now this process outputs an auxiliary $r\times r$
matrix $L$ whose eigenvalues are real and approximate  the  $r$
real eigenvalues of
 the companion matrix. 
It remains to apply the QR algorithm to the matrix $L$, 
 at the arithmetic cost $O(r^3)$ (cf. \cite[page 359]{GL96}),
 dominated if $r^3=O(kn\log (n))$.

The algorithm can be immediately applied to approximating 
all the real eigenvalues of a real nonsymmetric $n\times n$ matrix
by using $O(kn^3)$ arithmetic operations. 
We point out a direction to potential decrease of this complexity 
bound to $O((k+n)n^2)$ by means of similarity transformation to
 rank structured representation. 
Maintaining such representation would require additional research,
however, and we propose a distinct novel algorithm. It 
approximates all real eigenvalues of an $n\times n$
matrix by using $O(mn^2)$
arithmetic operations (see Algorithm \ref{alg6}),
and this bound decreases to $O(mn)$
for the companion and various generalized companion matrices.
Here $m$ denotes the number of iterations required for
the convergence
of the basic iteration of the algorithm.
Generally this number grows versus Algorithm
\ref{alg2} but remains reasonable for a large class of input matrices.

We engage, extend, and combine the number of efficient
methods
available for complex polynomial root-finding, particularly
the ones of  \cite{PZ11} and \cite{PQZ12},
but we also propose new techniques  and
 employ some old methods in novel and nontrivial ways.
 Our Algorithm
\ref{alg1} streamlines and substantially modifies
 \cite[Algorithm 9.1]{PZ11}
by avoiding the stage of root-squaring and the application of the
Cayley map. Some
techniques of Algorithm \ref{alg2} are implicit in \cite[Section
5]{PQZ12}, but we specify a  distinct iterative process, 
employ the Frobenius matrix algebra, extend the matrix sign iteration
to real eigen-solving, employ randomization and the QR algorithm,
 and include the initial acceleration by
scaling.
Our Algorithm \ref{alg3} naturally
extends Algorithms \ref{alg1} and \ref{alg2}, but we prove that
this extension is prone to the problems of numerical stability,
and our finding can be applied to 
the similar iterations of
\cite{BP96} and
\cite{C96} as well.
Algorithm \ref{alg4} can be linked to Algorithm \ref{alg1} 
and hence to \cite[Section 5]{PQZ12}, but 
incorporates some novel promising techniques.
Our simple recipe for real root-finding by means of combining the
root radii algorithm with Newton's iteration in Algorithm
\ref{algrrnt}  and even the extension of our  Algorithm \ref{alg2} to the
approximation of real eigenvalues of a real
nonsymmetric matrix are also  novel and 
promising.
Some of our algorithms take advantage of combining
the power of operating with matrices and polynomials (see Remarks
\ref{refr} and \ref{refr1}).
 Finding their deeper
  synergistic combinations 
is another natural research challenge, traced back to
\cite{P92} and \cite{BP94}.
 Our coverage of the complex plane geometry and
various rational transformations of the variable and 
the roots 
 can be of  independent interest.


Hereafter ``flops" stands for ``arithmetic operations",
``lc$(p)$" stands for ``the leading coefficient of $p(x)$".
$D(X, r)=\{x: |x-X|\le r\}$ and $C(X, r)=\{x: |x-X|=r\}$ denote
 a disc and a circle  on the complex plane, respectively.
We  write $||\sum_i v_ix^i||_q=(\sum_i|v_i|^q)^{1/q}$ for $q=1,2$
and $||\sum_i v_ix^i||_{\infty}=\max_i|v_i|$.
A function is in $\tilde O(f(bc))$ if it is in $O(f(bc))$
up to polylogarithmic factors in $b$ and $c$.
 agcd$(u,v)$ denotes an {\em approximate
greatest common divisor} of two polynomials
$u(x)$ and $v(x)$ (see \cite{BB10}
on definitions and algorithms).


\section{Some Basic Results for Polynomial Computations}\label{s2}


\subsection{Mappings of the Variables and the Roots}\label{smpg}


Some important maps of the roots of a polynomial can be computed
at a linear or nearly linear cost.
 
\begin{theorem}\label{thshsc} ({\em  Root Inversion, Shift and Scaling}, cf. \cite{P01}.)
 Given a polynomial $p(x)$ of (\ref{eqpoly}) and two scalars $a$ and $b$, we can
compute the coefficients of the polynomial $q(x)=p(ax+b)$ by using
$O(n \log (n))$ flops. We need only $2n-1$ flops if $b=0$.
Reversing a polynomial inverts all its roots involving no flops,
that is, $p_{\rm rev}(x)=x^np(1/x)=\sum_{i=0}^np_ix^{n-i}=p_n\prod_{j=1}^n(1-xx_j)$.
\end{theorem}

\begin{theorem}\label{thdnd} ({\em Root Squaring}, cf. \cite{H59}.)
(i) Let a polynomial $p(x)$ of (\ref{eqpoly}) be monic.
Then $q(x)=(-1)^np(\sqrt{x})p(-\sqrt{x})=\prod_{j=1}^n(x-x_{j}^2)$, and (ii) one can
evaluate $p(x)$ at the $k$-th roots of unity for $k>2n$ and then
interpolate to $q(x)$ by using $O(n\log (n))$ flops.
\end{theorem}

Recursive root-squaring is prone to numerical problems because the coefficients of the iterated polynomials very quickly span many orders of magnitude.  One can counter this deficiency
by using a special tangential representation of the coefficients and intermediate results (cf. \cite{MZ01}).

\begin{theorem}\label{thcl} ({\em Cayley Maps.})
The maps
$y=(x -\sqrt{-1})/(x+\sqrt{-1})$
and $x=\sqrt{-1}(y +1)/(y-1)$ send the real axis $\{x:~x~{\rm is~real}\}$
onto the unit circle $C(0,1)=\{y:~|y|=1\}$,  and vice versa.
\end{theorem}

\begin{theorem}\label{thmb} ({\em M{\"o}bius Maps.})
(i) The maps $\widehat y=\frac{1}{2}(\widehat x+1/\widehat x)$, 
$\widehat x=\widehat y\pm \sqrt{\widehat y^2-1}$ 
and $y=\frac{1}{2}(x-1/x)$, $x=y\pm \sqrt{y^2+1}$ send
the unit circle $C(0,1)=\{x:~|x=1|\}$ into the line intervals
$[-1,1]=\{\widehat y:~\Im \widehat y=0,~-1\le \widehat y\le 1\}$
and $[-\sqrt{-1},\sqrt{-1}]=\{y:~\Re y=0,~-1\le y\sqrt{-1}\le 1\}$, and vice versa. (ii) Write
$\widehat y=\frac{1}{2}(\widehat x+1/\widehat x)$, 
$\widehat y_{j}=\frac{1}{2}(\widehat x_j+1/\widehat x_j)$,
$y=\frac{1}{2}(x-1/x)$, and  $y_{j}=\frac{1}{2}(x_j-1/x_j)$, for
$j=1,\dots,n$. Then 
$\widehat q(\widehat y)=p(\widehat x)p(1/\widehat x)=
\widehat q_{n}\prod_{j=1}^n(\widehat y-\widehat y_{j})$ (cf.
\cite[equation (14)]{BP96})
and $q(y)=p(x)p(-1/x)=q_n\prod_{j=1}^n(y-y_{j})$. (iii) Given a polynomial $p(x)$ of
(\ref{eqpoly}), one can interpolate to the polynomials
$\widehat q(y)$ and $q(y)$ by using $O(n\log(n))$ flops.
\end{theorem}
\begin{proof} Verify part (i) immediately. Parts (ii) and (iii)
are proved in \cite[Section 2]{BP96} assuming 
$\widehat y=\frac{1}{2}(\widehat x+1/\widehat x)$ and 
$\widehat y_{j}=\frac{1}{2}(\widehat x_j+1/\widehat x_j)$, for
$j=1,\dots,n$.
The proof is readily extended to the case of
$y=\frac{1}{2}(x-1/x)$ and  $y_{j}=\frac{1}{2}(x_j-1/x_j)$, for
$j=1,\dots,n$
(e.g.,  
$\frac{1}{2}(x+1/x)=\cos (\phi)$ and $\frac{1}{2}(x-1/x)=\sin (\phi)$
for $x=\exp(\phi~\sqrt{-1})$ and real $\phi$). 
  \cite[Section 2]{BP96} reduces the computations to 
the evaluation and interpolation at the Chebyshev nodes,
and then the application of the algorithms of \cite{P89} or  \cite{P98}
yields the claimed  cost bounds, even though the paper \cite{P89} slightly overestimates
 the cost bound of its interpolation algorithm.
\end{proof}

\begin{theorem}\label{thsqnit} ({\em Error of  M{\"o}bius Iteration}.)
Fix a complex $x=x^{(0)}$ and
   define the iterations
\begin{equation}\label{eq20l0}
x^{(h+1)}=\frac{1}{2}(x^{(h)}+
1/x^{(h)})
~{\rm and}~\gamma=\sqrt {-1}~{\rm for}~h=0,1,\dots,
\end{equation}

\begin{equation}\label{eq20l}
x^{(h+1)}=\frac{1}{2}(x^{(h)}-1/x^{(h)})~{\rm and}~\gamma=1~{\rm for}~h=0,1,\dots
\end{equation}
The values $x^{(h)}\gamma$ are real for all $h$
if $x^{(0)}\gamma$ is real.
Otherwise
$|x^{(h)}-{\rm sign}(x)\sqrt {-1}/\gamma|\le \frac{2\tau^{2^{h}}}{1-\tau^{2^h}}$ for
$\tau=|\frac{x-{\rm sign}(x)}{x+{\rm sign}(x)}|$
and
$h=0,1,\dots$
\end{theorem}
\begin{proof} 
The bound is from \cite[page 500]{BP96})
under (\ref{eq20l0}), that is,
for $\gamma=\sqrt {-1}$,
and is readily extended to
the case  of (\ref{eq20l}),
that is, for $\gamma=1$.
\end{proof}





\subsection{Root Radii Approximation and Proximity Tests}\label{srrd}


\begin{theorem}\label{thrrd} ({\rm Root Radii Approximation}.) 
Assume
a polynomial
 $p(x)$ of (\ref{eqpoly}) and
two real scalars $c>0$ and $d$. Define the $n$ {\em root radii}
$r_j=|x_{k_j}|$ for $j=1,\dots, n$, distinct $k_1,\dots,k_n$, and
 $r_1\ge r_2\ge \cdots\ge r_n$.
Then,  by using
$O(n \log^2 (n))$ flops, we can compute $n$ approximations $\tilde r_j$ such that $\tilde
r_j\le r_j\le (1+c/n^d)\tilde r_j$, for $j=1,\dots, n$.
\end{theorem}
\begin{proof} (Cf. \cite{S82},  \cite[Section 4]{P00}, \cite[Section 15.4]{MP13}.) At first fix a sufficiently large integer $k$
and  apply $k$ times the root-squaring of Theorem \ref{thdnd},
by using $O(kn\log(n))$ flops. Then 
apply the algorithm of \cite{S82} to approximate all root radii $r_j^{(k)}=r_j^{2^k}$,
$j=1,\dots,n$,
of the output polynomial $p_k(x)$ within a factor
of $2n$ by using $O(n)$ flops. Hence 
the root radii $r_1,\dots,r_n$ are approximated
within a factor of $(2n)^{1/2^k}$, which is $1+c/n^d$ for  $k$
of order $\log(n)$.
\end{proof}

Alternatively one can estimate the root radii by
applying Gerschg{\"orin}
theorem  to the companion matrix of a polynomial $p(x)$,
defined in Section \ref{scmpn} (see \cite{C91})
or by using heuristic methods (see  \cite{BRa}).
Next we approximate the largest root radius  $r_1$ of $p(x)$ at a lower cost.
Applying the same algorithms to 
the reverse polynomial $p_{\rm rev}(x)$
yields the smallest root radius $r_n$  of $p(x)$ (cf. Theorem \ref{thshsc}).


\begin{theorem}\label{thextrrrd} (See \cite{VdS70}.)
Assume a polynomial $p(x)$ of  (\ref{eqpoly}). Write
$r_1=\max_{j=1}^{n}|x_j|$, $r_n=\min_{j=1}^{n}|x_j|$, and
$\gamma^+=\max_{i=1}^{n}|p_{n-i}/p_n|$.
Then $\gamma^+/n\le r_1\le 2\gamma^+$.
\end{theorem}


\begin{theorem}\label{threfext} (See \cite{P01a}.) 
For  $\epsilon=1/2^b>0$,
one only needs $a(n,\epsilon)=O(n+b\log (b))$ flops to compute an
approximation $r_{1,\epsilon}$ to the root $r_1$ radii of $p(x)$
such that $r_{1,\epsilon}\le r_1\le 5(1+\epsilon)r_{1,\epsilon}$.
In particular, $a(n,\epsilon)=O(n)$ for $b=O(n/\log (n))$, 
and $a(n,\epsilon)=O(n\log (n))$ for $b=O(n)$.
\end{theorem} 
 The latter theorem and the heuristic proximity test 
 below can be applied 
even where a polynomial $p(x)$ is defined by a black box subroutine for its evaluation
rather than by its coefficients.

 By shifting and scaling
the variable (cf. Theorem  \ref{thshsc}),
 we can 
move
 all roots of $p(x)$ into a fixed disc, e.g., $D(0,1)=\{x:~|x|\le 1\}$.
The smallest root radius $r_n$ of the polynomial $q(x)=p(x-c)$
for a complex point $c$  denotes the minimum distance
of the roots from this point. Approximation of this distance is called
{\em proximity test} at the point $c$. Besides Theorems 
\ref{thextrrrd} and
\ref{threfext}, one can apply heuristic proximity test at 
a  point $c$
 by means of Newton's iteration,
 \begin{equation}\label{eqnewt}
y_0=c,~y^{(h+1)}=y^{(h)}-p(y^{(h)})/p'(y^{(h)}),~h=0,1,\dots
\end{equation}  
If $c$ approximates a simple isolated root, the iteration 
 refines this approximation very fast.


\begin{theorem}\label{thren} (See \cite[Corollary~4.5]{R87}.)
  Suppose both discs $D(y_0, r)$ and $D(y_0, r/s)$ for $s\ge 5n^2$
  contain a single simple root $y$
  of 
  a polynomial $p=p(x)$ of  (\ref{eqpoly}).
  Then  Newton's
  iteration (\ref{eqnewt})
  converges to this root 
  right from the start, so that  $|y_k-y|\le 8|y_0-y|/2^{2^k}$.
\end{theorem}

By exploiting the 
correlations between the coefficients of a polynomial 
and the power sums of its roots, the paper  \cite{PT14b} had weakened
the above assumption that $s\ge 5n^2$  to 
allow any constant $s>1$. 
By recursively squaring the variable and the roots
 $O(\log (n))$ times (as in the proof of Theorem \ref{thrrd}), one can allow any $s$ 
below $1+c/n^d$, for any pair of real constants $c>0$ and $d$.


\subsection{Two Auxiliary Algorithms for the First Polynomial Root-finder}\label{sauxalg}


\begin{theorem}\label{thmlg} ({\em Root-finding Where All Roots Are Real}).
The  modified Laguerre algorithm of \cite{DJLZ97} converges
to all roots of a polynomial  $p(x)$ of (\ref{eqpoly}) right from the start,
 uses $O(n)$ flops per iteration, and therefore
approximates all $n$ roots within $\epsilon =1/2^b$
by using $O(\log (b))$ iterations and
 performing $\tilde O(n\log (b))$ flops overall.
 This asymptotic cost bound is optimal
and is also supported by the
 alternative algorithms of \cite{BT90} and \cite{BP98}.
\end{theorem}

\begin{theorem}\label{thsc} ({\em Splitting a Polynomial into Two Factors
Over a Circle}, cf. \cite{S82} or \cite[Chapter 15]{MP13}.)
 Suppose a polynomial $t(x)$ of degree $n$
 has $r$ roots
inside the
 circle $C(0,\rho)$ and
$n-r$ roots outside the  circle $C(0,R)$
for $R/\rho\ge  1+ 1/n$.  Let
$\epsilon=1/2^b$ for $b\ge n$.
(i) Then by  performing $O((\log^2(n)+\log(b))n\log(n))$ flops
(that is, $O(n\log^3(n))$ flops for $\log(b)=O(\log^2(n))$),
 with a precision of $O(b)$ bits, we can
 compute two polynomials $\tilde f$ and $\tilde g$
such that $||p-\tilde f\tilde g||_q\le \epsilon ||p||_q$
for $q=1,2$ or $\infty$,
the polynomial $\tilde f$ of degree $r$  has $r$ roots inside the circle $C(0,1)$,
and the polynomial $\tilde g$ of degree $n-r$  has $n-r$ roots
outside this circle. (ii) By recursively squaring the variable and the roots
 $O(\log (n))$ times (by using $O(n\log^2(n))$ flops), one can extend the result of part (i)
to the case where $R/\rho\le 1+1/n$, for any pair of
positive constants $c$ and $d$.
\end{theorem}


\section{Root-finding As Eigen-solving and Basic Definitions and Results for Matrix Computations}\label{s3}

\subsection{Some Basic Definitions for Matrix Computations}\label{sfnd}



$M^T=(m_{ji})_{i,j=1}^{n,m}$ is the transpose of a matrix $M=(m_{ij})_{i,j=1}^{m,n}$.
$M^H$ is its Hermitian  transpose.
$I=I_n=({\bf e}_1~|~{\bf e}_2~|\ldots~|~{\bf e}_{n})$ is the $n\times n$ identity matrix, whose columns
are the $n$ coordinate vectors ${\bf e}_1,~{\bf e}_2,\ldots,{\bf e}_{n}$.
$\diag(b_j)_{j=1}^s=\diag(b_1,\dots,b_s)$ is the $s\times s$  diagonal matrix
with the diagonal entries $b_1$, $\dots$, $b_s$.
$\mathcal R(M)$ is the range of
a matrix $M$,
that is, the linear space
generated by its columns. A
 matrix of full column rank is a {\em matrix basis} of its range.

A matrix $Q$ is
{\em unitary } if $Q^HQ=I$
 or $QQ^H=I$. 
 $(Q,R)=(Q(M),R(M))$ for an $m\times n$ matrix $M$ of rank $n$
denotes a unique pair of unitary $m\times n$ matrix $Q$ and
 upper triangular $n\times n$ matrix $R$ such that $M=QR$
and all diagonal entries of the matrix $R$
are positive  \cite[Theorem 5.2.2]{GL96}.

$M^+$ is the Moore--Penrose pseudo inverse of $M$ \cite[Section 5.5.4]{GL96}.
An $n\times m$ matrix $X=M^{(I)}$ is a left 
inverse of an $m\times n$  matrix $M$ if $XM=I_n$
 $M^{(I)}=M^+$ for a matrix
$M$ of full rank. $M^{(I)}=M^H$ for a  unitary matrix $M$.
 $M^{(I)}=M^+=M^{-1}$ for
 a nonsingular matrix $M$.

\begin{definition}\label{defeig}
$\mathcal S$
is the {\em invariant subspace} of a
square matrix $M$
if $M\mathcal S=\{M{\bf v}:{\bf v}\in \mathcal S\}\subseteq \mathcal S$.
A scalar  $\lambda$ is an
{\em eigenvalue}  of a matrix $M$
 associated with an {\em eigenvector}  ${\bf v}$ if
$M{\bf v}=\lambda{\bf v}$.
All eigenvectors associated with an
eigenvalue $\lambda$ of $M$ form an eigenspace
$\mathcal S(M,\lambda)$, which is an invariant space.
Its dimension $d$
 is the {\em geometric multiplicity} of $\lambda$.
The eigenvalue is simple if its multiplicity is $1$. The set
$\Lambda(M)$ of all eigenvalues of a matrix $M$ is called its {\em
spectrum}.
\end{definition}


\subsection{The Companion Matrix and the Frobenius Algebra}\label{scmpn}


 $$C_p=\begin{pmatrix}
        0   &       &       &   & -p_0/p_n \\
        1   & \ddots    &       &   & -p_1/p_n \\
            & \ddots    & \ddots    &   & \vdots    \\
            &       & \ddots    & 0 & -p_{n-2}/p_n \\
            &       &       & 1 & -p_{n-1}/p_n \\
    \end{pmatrix}
$$ 
is the {\em companion matrices} of the polynomial $p(x)$ of
(\ref{eqpoly}).
$p(x)=c_{C_p}(x)=\det(xI_n-C_p)$ is the  {\em
characteristic polynomial} of $p(x)$. Its roots form the spectrum of
$C_p$, and so real root-finding for the polynomial $p(x)$ turns into real
eigen-solving for the  matrix $C_p$. 


\begin{theorem}\label{facfr} ({\em The Cost of Computations in the Frobenius Matrix Algebra}, 
cf. \cite{C96} or \cite{P05}.)
The companion matrix $C_p\in \mathbb C^{n\times n}$ of a polynomial $p(x)$ of (\ref{eqpoly})
generates Frobenius matrix algebra.
One needs  $O(n)$ flops for  addition,
$O(n\log (n))$ flops for multiplication, and $O(n\log^2 (n))$ flops
for inversion  in  this algebra. One
needs $O(n\log (n))$ flops to multiply a matrix in this algebra
  by a vector.
\end{theorem}


\subsection{Decreasing the Size of an Eigenproblem}\label{seig}


Next we reduce eigen-solving for the matrix $C_p$ to
the study of its invariant space generated by the $r$ eigenspaces
associated with the $r$ real eigenvalues. The following theorem is
basic for this step.


\begin{theorem}\label{thsubs} ({\em Decreasing the Eigenproblem Size to the Dimension of
an Invariant Space}, cf.
 \cite[Section 2.1]{W07}.)
Let $U\in \mathbb C^{n\times r}$,
$\mathcal R(U)=\mathcal U$, and $M\in\mathbb C^{n\times n}$.
 Then (i) $\mathcal U$ is an invariant space of $M$
if and only if there exists a matrix $L\in\mathbb C^{k\times k}$
such that $MU=UL$ or equivalently if and only if $L=U^{(I)}MU$,
(ii) the matrix $L$ is unique (that is, independent
of the choice of the left inverse $U^{(I)}$) if $U$
is a matrix basis for the space $\mathcal U$,
(iii) $\Lambda(L)\subseteq \Lambda(M)$,
(iv) $L=U^HMU$ if $U$ is a unitary matrix, and
(v) $MU{\bf v}=\lambda U{\bf v}$
if $L{\bf v}=\lambda {\bf v}$.
\end{theorem}


By virtue of the following theorem,
 a matrix function shares its
 invariant spaces with the matrix $C_p$, and so 
we can facilitate the computation of the desired invariant space
of $C_p$ if we reduce the task to the case of an appropriate matrix
function, for which the solution is simpler.




\begin{theorem}\label{thsmf} ({\em The Eigenproblems for a Matrix
and Its Function}.) Suppose $M$ is a square matrix, a rational
function $f(\lambda)$ is defined  on its spectrum, and $M{\bf
v}=\lambda{\bf v}$. Then (i) $f(M){\bf v}=f(\lambda){\bf v}$. (ii)
Let $\mathcal U$ be the eigenspace of the matrix $f(M)$ associated
with its eigenvalue $\mu$. Then this is an invariant space of the
matrix $M$ generated by its eigenspaces associated with all its
eigenvalues $\lambda$ such that $f(\lambda)=\mu$. (iii) The space
$\mathcal U$ is associated with a single eigenvalue of $M$ if
$\mu$ is a simple eigenvalue of $f(M)$.
\end{theorem}


\begin{proof}
We readily verify part (i), which implies parts (ii) and (iii).
\end{proof}

Suppose we have computed a matrix basis
 $U\in \mathbb C^{n\times r}$  for
an  invariant space  $\mathcal U$ of  a matrix function $f(M)$ of
an $n\times n$ matrix $M$. By virtue of Theorem \ref{thsmf},
this is  a matrix basis of an invariant space of the matrix $M$.
We can first compute a left inverse $U^{(I)}$
or the
orthogonalization $Q=Q(U)$
and then
approximate the eigenvalues of $M$
associated with this eigenspace as the eigenvalues of the
$r\times r$ matrix $L=U^{(I)}MU=Q^HMQ$
(cf. Theorem \ref{thsubs}).

Given
 an approximation $\tilde \mu$ to a simple eigenvalue
 of a matrix function $f(M)$,  we can compute an approximation $\tilde {\bf u}$
to an eigenvector
${\bf u}$ of the matrix $f(M)$ associated with this eigenvalue, recall from 
part (iii) of Theorem
\ref{thsmf} that this is also an eigenvector of the matrix $M$,
associated with its simple eigenvalue, and
approximate this  eigenvalue by the Rayleigh Quotient
$\frac{\tilde {\bf u}^TM\tilde {\bf u}}{\tilde {\bf u}^T\tilde {\bf u}}$.


\subsection{Some Maps in the Frobenius Matrix Algebra}\label{smbms3}


Part (i) of
Theorem \ref{thsmf} implies that 
for a polynomial $p(x)$ of (\ref{eqpoly}) and a rational function
$f(x)$ defined on the set $\{x_i\}_{i=1}^n$ of its roots, the
rational matrix function $f(C_p)$ has the spectrum
$\Lambda(f(C_p))=\{f(x_i)\}_{i=1}^n$. In particular, the maps
$$C_p\rightarrow C_p^{-1},~C_p\rightarrow aC_p+bI,~
C_p\rightarrow C_p^2,~C_p\rightarrow \frac{C_p+C_p^{-1}}{2},~{\rm and}~
C_p\rightarrow \frac{C_p-C_p^{-1}}{2}$$
induce the maps of the eigenvalues of the matrix $C_p$, and thus induce
 the maps of the roots of its characteristic polynomial $p(x)$
given by the equations
$$y=1/x,~y=ax+b,~y=x^2,~y=0.5(x+1/x),~{\rm and}~y=0.5(x-1/x),$$
respectively.
By using the reduction modulo $p(x)$,
define the five dual maps
\begin{eqnarray*}
y=(1/x) \mod p(x),~y=ax+b\mod p(x),~
y=x^2\mod p(x),\\
~y=0.5(x+1/x)\mod p(x),~
{\rm and}~y=0.5(x-1/x)\mod p(x),
\end{eqnarray*}
where $y=y(x)$ denotes polynomials.
Apply the two latter maps recursively, to
define two iterations with polynomials modulo $p(x)$ as follows,
$y_0=x$, $y_{h+1}=0.5(y_h+1/y_h)\mod p(x)$ (cf. (\ref{eq20l0})) and
$y_0=x,~y_{h+1}=0.5(y_h-1/y_h)\mod p(x)$ (cf. (\ref{eq20l})), $h=0,1,\dots$.
More generally, define the iteration
$y_0=x$, $y_{h+1}=ay_h+b/y_h\mod p(x)$, $h=0,1,\dots$,
 for any pair of scalars $a$ and $b$. 


\section{Real Root-finders}\label{s4}


\subsection{M{\"o}bius Iteration}\label{smbms0}


Theorem \ref{thsqnit} implies that right from the start of
iteration (\ref{eq20l}) the values $x^{(h)}$ converge fast to $\pm \sqrt {-1}$
unless the initial value $x^{(0)}$ is real,
in which case  all iterates $x^{(h)}$ are real. It follows that
right from the start
the values $y^{(h)}=(x^{(h)})^2+1$ converge fast  to 0
unless  $x^{(0)}$ is real, whereas all values $y^{(h)}$ are real and
exceed 1  if 
$x^{(0)}$ is real. Write $t_h(y)=\prod_{j=1}^n(y-(x_j^{(h)})^2-1)$
 and  $v_h(y)=\prod_{j=1}^r(y-(x_j^{(h)})^2-1)$
for $h=1,2,\dots$.
The roots of the polynomials  $t_h(y)$ and $v_h(y)$ are the images of
all roots and of the real roots of the polynomial $p(x)$ of (\ref{eqpoly}), respectively,
produced by
the composition of the maps (\ref{eq20l}) and
$y^{(h)}=(x^{(h)})^2+1$. Therefore
 $t_h(y)\approx y^{2s}v_h(y)$ for large integers $h$
where the polynomial $v_h(y)$ has degree $r$ and 
has exactly $r$ real roots, all 
 exceeding 1, and so
 for large integers $h$, 
the sum of the $r+1$
leading terms of the polynomial $t_h(y)$ closely approximates
the polynomial $y^{2s}v_h(y)$. (To verify that the $2s$
trailing coefficients nearly vanish, we need just $2s$
comparisons.) The above argument shows correctness of the
following algorithm. (One can similarly apply and analyze
 iteration (\ref{eq20l0}).)


\begin{algorithm}\label{alg1} {\rm M{\"o}bius iteration for real root-finding.}

    \item{\textsc{Input:}} two integers $n$ and $r$, $0<r<n$, and
    the coefficients of a polynomial $p(x)$ of equation (\ref{eqpoly})
    where $p(0)\neq 0$.

    \item{\textsc{Output:}}
    approximations to the real roots $x_1,\dots,x_r$ of the polynomial $p(x)$.


    \item{\textsc{Computations:}}
         \begin{enumerate}
         \item
          Write $p_0(x)=p(x)$ and recursively compute the polynomials $p_{h+1}(y)$
      such that $p_{h+1}(y) =p_h(x)~p_h(-1/x)$
     for $y=(x-1/x)/2$ and $h=0,1,\dots$
     (Part (ii) of Theorem  \ref{thmb} combined with Theorem \ref{thsqnit} 
         defines the images of the real and nonreal roots
         of the polynomial $p(x)$ for all $h$.)

         \item
         Periodically, at some selected Stages $k$,
compute the polynomials
  $$t_h(y)=(-1)^nq_k(\sqrt {y+1})~q_h(-\sqrt{y+1})$$ 
where $q_k(y)=p_k(y)/{\rm lc}(p_k)$ (cf. Theorems
\ref{thshsc} and \ref{thdnd}). When the integer $k$ becomes large
enough, so that $2s$ trailing coefficients of the polynomial
$q_k(x)$ vanish or nearly vanish, delete these coefficients
and divide the resulting polynomial by $x^{2s}$,
to obtain a polynomial $v_k(x)$ of degree $r$, which
 is an approximate
 factor  of the polynomial $t_k(x)$ 
and has $r$
real  roots on the ray $\{x:~x\ge 1\}$.

     \item
     Apply one of the algorithms of  \cite{BT90}, \cite{BP98}, and \cite{DJLZ97}
         (cf. Theorem \ref{thmlg}) to
     approximate the $r$ roots of the polynomial $v_k(x)$.

     \item
Extend the descending process
from \cite{P95}, \cite{P02} and \cite
{BP96} to recover approximations to the $r$ roots $x_1$,
$\dots,x_r$ of the polynomial $p_0(x)=p(x)$. At first, 
having the $r$ roots $w_j$  of the polynomial  $v_k(x)$ approximated, 
compute the
 $2r$ values $\pm \sqrt {w_j-1}$, $j=1,\dots,r$.
 Then select among them the $r$ values $x_j^{(k)}$, $j=1,\dots,r$,
by applying one of the proximity tests of Section \ref{srrd}
 to 
the polynomial $q_k(y)$ at all of these $2r$ values.
(The $r$ selected values
 approximate the $r$ common real roots of the polynomials
$q_k(y)$ and $p_k(y)$.)
Compute the  $2r$ values $x_j^{(k)}\pm \sqrt{(x_j^{(k)})^2+1}$, $j=1,\dots,r$.
By virtue of part (i) of Theorem  \ref{thmb}, $r$ of these values approximate 
the $r$ real roots of the polynomial
 $p_{k-1}(x)$. Select these approximations by applying one of
the proximity tests of  Section \ref{srrd}
 to 
the polynomial  $p_{k-1}(x)$ at all of the $2r$ candidate values.
Continue  recursively to descend down to the $r$ real roots of $p_0(x)=p(x)$.
The process is not ambiguous because only
$r$ roots of the polynomial $p_{h}(x)$ are real for each
$h$, by virtue of Theorem \ref{thsqnit}.

    \end{enumerate}
    \end{algorithm}

 Like  lifting Stage 1, descending Stage 4 involves order of $kn\log (n)$ flops,
which also bounds the overall cost of performing the algorithm.

\begin{remark}\label{rerfn} ({\rm Refinement by means of Newton's iteration.})
For every $h$, $h=k,k-1,\dots,0$, one can apply Newton's iteration
$x_{j,i+1}^{(h)}=x^{(h)}-p(x_{j,i}^{(h)})/p'(x_{j,i}^{(h)})$, $h=0,1,\dots$,
$i=0,1,\dots,l$,
concurrently at the $r$ approximations $x_j^{(h)}$, $j=1,\dots,r$,
to the $r$ real roots of the polynomial $p_h(x)$. 
We can perform  
$l$ iteration loop by using $O(nl\log^2(r))$ flops,
that is $O(n\log^2(r))$ flops per loop
(cf. \cite[Section 3.1]{P01}),  adding this to
the overall arithmetic cost of order $kn\log(n)$
for performing the algorithm.  We can 
perform
the proximity tests of Stage 4
of the algorithm 
by applying Newton's iteration at all $2r$
candidate approximation points. Having selected 
$r$ of them, we 
can continue applying the iteration at these points,
 to refine the approximations.
\end{remark}

\begin{remark}\label{redgn} ({\rm Countering Degeneracy.})
If $p(0)=p_0=\dots=p_m=0\neq p_{m+1}$, then we should output the
real root $x_0=0$ of multiplicity $m$ and apply the algorithm to
the polynomial $p(x)/x^m$ to approximate the other real roots.
Alternatively we can apply the algorithm to the polynomial
$q(x)=p(x-s)$ for a shift value $s$ such that $q(0)\neq 0$. With
probability 1,  $q(0)\neq 0$ for Gaussian random variable $s$, 
but we can approximate the root radii of the polynomial
$p(x)$ (cf. Theorem \ref{thrrd})
 and then deterministically find  a 
 scalar $s$ such
that $q(x)$ has no roots near 0.  
\end{remark}

\begin{remark}\label{refr2} ({\rm Saving the Recursive Steps of Stage 1.})
The first goal of the algorithm is the computation of  
a polynomial $v_k(x)$ of
degree $r$ that has $r$ real roots and is 
an approximate factor  of the polynomial $t_k(x)$.
 If the assumptions of Theorem \ref{thsc}are satisfied for $t(x)=t_k(x)$
for a smaller integer $k$
we can compute  a polynomial
 $v_k(x)$ for this $k$ decreasing the overall computational
 cost.
For a fixed $k$ we can verify the assumptions 
by using $O(n\log^2(n))$ flops
(by applying the root radii
algorithm of Theorem \ref{thrrd}),
and so it is not too costly to test  even all integers
$k$ in the range, unless the range is large.
BY using the binary search
 for the minimum integer $k$
 satisfying  Theorem \ref{thsc}, 
 we would need only $O(\log (n))$ tests, that is,
$O(n\log^3(n))$ flops.
\end{remark}

\begin{remark}\label{renrlr} ({\rm Handling the Nearly Real Roots.})
The integer parameter
$k$ and the overall arithmetic cost of performing
the algorithm are large
if the value
$2^{-d}=\min_{j=r+1}^n |\Im x_j|$ is small. 
We can counter this deficiency by
splitting out
from the polynomial $t_k(x)$
its factor $v_{k,+}(x)$ 
of degree $r_+>r$ that has  $r_+$ real and nearly real roots
if the other nonreal roots lie sufficiently far from the real axis.
Our convergence analysis and
the recipes for splitting out the factor $v_k(x)$
(including the previous remark)
can be readily extended.  
If the integer $r_+$ is small,
we can compute all the $r_+$ roots of
the polynomial $v_{k,+}(x)$ 
at a low cost and then select the $r$ real roots among them.) 
Even if the integer $r_+$ is large, 
but  all of $r_+$ roots of the polynomial $v_{k,+}(x)$
lie on or close enough to the real axis,
we can approximate these roots at a low cost by
applying the modified Laguerre algorithm of \cite{DJLZ97}.
\end{remark}

\begin{remark}\label{renrr} ({\rm The Number of Real Roots.})
We assume that we know the number $r$ of the real roots
(e.g., supplied by noncostly algorithms of computer
algebra), but we can compute
this number as by-product of Stage 2, and similarly for our other
algorithms. With a proper policy we can
compute the integer $r$ by
testing  at most $2+2\lceil \log_2(r)\rceil$
 candidates in the range $[0,2r-1]$.
\end{remark}


\subsection{An Extended Matrix Sign Iteration}\label{smbms1}


The known upper bounds on the condition numbers of
the roots of the polynomials $p_{k}(y)$  grow exponentially
 as $k$ grows large (cf. \cite[Section 3]{BP96}).
If the bounds are actually sharp,
Algorithm \ref{alg1}
is prone to numerical stability problems already for
moderately large integers $k$.
We can avoid this potential deficiency by replacing the 
iteration of Stages 1 and 2 by the dual matrix iteration
\begin{equation}\label{eqmsn}
Y_0=C_p,~Y_{h+1}=0.5(Y_h-Y_h^{-1})~{\rm for}~ h=0,1,\dots.
\end{equation}
It extends the  {\em matrix sign} iteration
$\widehat Y_{h+1}=0.5(\widehat Y_h+\widehat Y_h^{-1})$ for $h=0,1,\dots$
(cf.  (\ref{eq20l0}), (\ref{eq20l}), part (ii) of our Theorem \ref{thmb}, and \cite{H08}) and maps the eigenvalues 
of the matrix $Y_0=C_p$ according to  (\ref{eq20l}). So
 Stage 1 of
Algorithm \ref{alg1} maps
the characteristic polynomials
of the above matrices $Y_h$. Unlike the case of the latter map,
working with matrices enables us to recover the desired real
eigenvalues of the matrix $C_p$ by means of our recipes of Section
\ref{s3}, without recursive descending.  

\begin{algorithm}\label{alg2} {\rm Matrix sign iteration modified for real eigen-solving.}

    \item{\textsc{Input and Output} as in Algorithm \ref{alg1}, except
    that FAILURE can be output with a probability close to 0.}

    \item{\textsc{Computations:}}
         \begin{enumerate}
         \item
         Write $Y_0=C_p$ and recursively compute the matrices $Y_{h+1}$ of
(\ref{eqmsn})
 for $h=0,1,\dots$
($2s$ eigenvalues of
the matrix $Y_h$ converge to $\pm \sqrt {-1}$ as $h\rightarrow \infty$,
whereas its $r=n-2s$ other eigenvalues are real for all $h$, 
by virtue of Theorem \ref{thsqnit}.)

         \item
         Fix a 
	 sufficiently 
	 large integer $k$ and compute the matrix $Y=Y_k^2+I_n$.
         (The map $Y_0=C_p\rightarrow Y$ sends all nonreal eigenvalues of $C_p$
         into a small neighborhood of the origin 0 and sends all real
         eigenvalues of $C_p$ into the ray $\{x:~x\ge 1\}$.)

         \item
          Apply the randomized algorithms of \cite{HMT11}
         to compute the  numerical rank of the matrix $Y$.
  The rank is at least $r$, and if it exceeds $r$, then  go back to Stage 1.
     If it is equal to $r$, then generate a  standard Gaussian random $n\times r$ matrix $G$
         and compute the matrices $H=YQ(G)$ and $Q=Q(H)$.
   (The analysis of
preprocessing with  Gaussian random multipliers in
 \cite[Section 4]{HMT11}, \cite[Section 5.3]{PQYa} shows that,
 with a probability close to 1, the columns of the matrix
$Q$ closely approximate a unitary basis of the invariant space of
the matrix $Y$ associated with its $r$ absolutely largest
eigenvalues, which are the images of the real eigenvalues of the
matrix $C_p$. Having this approximation is equivalent to having
 a small upper bound on the residual norm
 $||Y-QQ^HY||$  \cite{HMT11}, \cite{PQYa}.)
Verify the latter bound. If the verification
fails (which is  unlikely),
 output FAILURE and stop the computations.

     \item
Otherwise compute and output
approximations to the $r$ eigenvalues of the $r\times r$ matrix  $L=Q^HC_pQ$.
They approximate the real roots
of the polynomial $p(x)$. (Indeed, by virtue of  Theorem \ref{thsmf},
$Q$ is an approximate matrix basis for the invariant space of the matrix $C_p$
associated with its $r$ real eigenvalues. Therefore, by virtue of
 Theorem \ref{thsubs}, the $r$ 
eigenvalues of the matrix $L$ approximate the $r$ real eigenvalues of 
the matrix $C_p$.)
    \end{enumerate}
    \end{algorithm}


Stages 1 and 2 involve
$O(kn\log^2(n))$ flops by virtue of Theorem \ref{facfr}.
 This exceeds the estimate for
 Algorithm \ref{alg1} by a factor of $\log(n)$.
Stage 3  adds $O(nr^2)$ flops and  the cost $a_{rn}$ of
 generating $n\times r$ standard Gaussian random matrix.
The cost bounds are $O(r^3)$ at Stage 4 and 
$O((kn\log^2(n)+nr^2)+a_{rn}$ overall.


\begin{remark}\label{renrei} ({\rm Counting Real Eigenvalues.})
The binary search can produce  the number of real eigenvalues
as the numerical rank of the matrices $Y_k^2+I$
when this rank stabilizes.
\end{remark}


\begin{remark}\label{refracc} ({\rm Acceleration by Using Random Circulant Multiplier.})
We can decrease the  arithmetic cost of Stage 3 to $a_{n+r}+O(n\log(n))$
and can perform only $O(kn\log^2(n)+nr^2)+a_{r+n}$ flops overall
if we replace an $n\times r$ standard Gaussian random multiplier
by the product $\Omega CP$ where $\Omega$ and $C$ are $n\times n$ matrices,
 $\Omega$ is the matrix of the discrete Fourier transform,
$C$ is a random circulant matrix,
and $P$ is an $n\times l$ random permutation matrix,
for
a sufficiently large $l$ of order $r\log (r)$.
See  \cite[Section 11]{HMT11}, \cite[Section 6]{PQYa}
for the analysis and for the supporting probability estimates. They are
only
slightly less favorable than in the case of a  Gaussian random multiplier.
\end{remark}

\begin{remark}\label{refracclrb} ({\rm Acceleration by Means of Scaling.})
We can dramatically accelerate the initial convergence of
Algorithm \ref{alg2} by applying {\em determinantal scaling}  (cf.
\cite{H08}), that is, by computing the matrix $Y_1$ as follows,
$Y_{1}=0.5(\nu Y_0-(\nu Y_0)^{-1})$ for
$\nu=1/|\det(Y_0)|^{1/n}=|p_n/p_0|$,
 $Y_0=C_{p}$.
\end{remark}

\begin{remark}\label{refr} (Hybrid Matrix and Polynomial Algorithms.)
Can we
modify Algorithm \ref{alg2} to keep its advantages
but to decrease the arithmetic cost of its Stage 1
to the level $kn \log (n)$ of Algorithm \ref{alg1}?
Let us do this for a large class of input polynomials 
by applying a hybrid algorithm that
 combines the power of 
 Algorithms \ref{alg1} and  \ref{alg2}.
First note that we can replace iteration (\ref{eqmsn})
by any of the iterations
$Y_{h+1}=0.5(Y_h^{3}+3 Y_h)$ and
$Y_{h+1}=-0.125(3 Y_h^{5}+10 Y_h^{3}+15 Y_h)$ for $h=0,1,\dots$
provided that all or almost all nonreal roots of the polynomial $p(x)$
lie in the discs $D(\pm \sqrt{-1},1/2)$.
Indeed right from the start, the
iterations
send the nonreal roots lying
in these discs toward the two points $\pm \sqrt{-1}$
 with quadratic and cubic convergence rates,
respectively. (To prove this, extend the proof of \cite[Proposition
4.1]{BP96}.) Both iterations keep the real roots real,
involve no inversions, and use $O(n\log (n))$ flops per  loop.
These observations
suggest the following policy.
Perform the iterations of Algorithm \ref{alg1}  
as long as the outputs are not corrupted by rounding errors.
(Choose the number of iterations of Algorithm \ref{alg1} heuristically.)
For a large class of inputs,
the iterations (in spite of the above limitation on their number)
bring 
the images of the nonreal eigenvalues of $C_p$ into the basin of convergence 
 of the inversion-free matrix iterations above.
Now let $q_h(x)$ denote the auxiliary polynomial output
by Algorithm \ref{alg1}.
Then approximate its real roots
 by applying one of the  inversion-free iterations above
to its companion matrix $C_{q_h}$.
Descend from these roots
to the real roots of the polynomial $p(x)$
as in  Algorithms \ref{alg1}.
\end{remark}


\subsection{Numerical Stabilization of 
 the Extended Matrix Sign Iteration}\label{snstbz}


The images of nonreal eigenvalues of the matrix $C_p$ converge to
$\pm \sqrt {-1}$ in the iteration of Stage 1
of Algorithm
  \ref{alg2}, but the
images of some real eigenvalues of $C_p$ can come close to 0,
and then the next step of the
iteration would involve an ill conditioned matrix $Y_h$.
This would be a complication
unless we are applying an inversion-free variant of
the iteration of the previous remark.
We can
detect that the matrix $Y_h$ is ill conditioned
 by encountering difficulty in its numerical inversion or
by computing its smallest singular value
 (e.g., by applying the Lanczos algorithm
\cite[Proposition 9.1.4]{GL96}). 
In such cases we can try to avoid problems by
  shifting the matrix  (and its eigenvalues), that is,
by adding to or subtracting from the current matrix $Y_h$
the matrix $s I$ for a reasonably 
small positive scalar $s$. We can select this scalar
 by applying Theorem \ref{thrrd}, heuristic methods, or randomization.

Towards a more radical recipe, 
apply the following modification of Algorithm   \ref{alg2}.

\begin{algorithm}\label{alg2a} {\rm Numerical stabilization of an extended matrix sign iteration.}

    \item{\textsc{Input, Output} and Stages 3 and 4 of \textsc{Computations} are as in Algorithm \ref{alg2}, except that the input includes a small positive  scalar $\alpha$ 
such that no eigenvalues of the matrix $C_p$ have 
imaginary parts close to $\pm \alpha \sqrt {-1}$
(see Remark \ref{realph} below),
    the set of $r$ real roots $x_1,\dots,x_r$ of the polynomial $p(x)$
is replaced by the set of its $r_+$ roots having the imaginary parts in the
range
$[-\alpha,\alpha]$, and the integer $r$ is replaced by the integer $r_+$ throughout.}
  	
    \item{\textsc{Computations:}}
         \begin{enumerate}
         \item
         Apply Stage 1 of Algorithm \ref{alg2} to the two matrices 
         $Y_{0,\pm}=\alpha\sqrt {-1}~I\pm C_p$, producing  two sequences 
	of the matrices $Y_{h,+}$ and  $Y_{h,-}$ for $h=0,1,\dots$. 

         \item
         Fix a 
	 sufficiently 
	 large integer $k$ and compute the matrix $Y=Y_{k,+}+Y_{k,-}$.
 
    \end{enumerate}
    \end{algorithm}

Because of the assumed choice of $\alpha$, the matrices  $\alpha\sqrt {-1}~I\pm C_p$ have
no real eigenvalues, and so the images of all their eigenvalues, that is, the eigenvalues of the matrices
$Y_{k,+}$ and $Y_{k,-}$, converge to 
$\pm \sqrt {-1}$ as $k \rightarrow \infty$. Moreover, one can verify that the eigenvalues
of the matrix $Y_{k,+}+Y_{k,-}$ converge to 0 unless they are the images of the $r_+$ eigenvalues of the
matrix $C_p$ having the imaginary parts in the range $[-\alpha, \alpha]$. The latter   eigenvalues of the matrix  $Y_{k,+}+Y_{k,-}$
converge to $2\sqrt {-1}$. This shows correctness and numerical stability of 
Algorithm \ref{alg2a}.

The algorithm approximates the $r_+$ roots of $p(x)$ 
by using $O(kn\log^2(n)+nr_+^2)+a_{r_+n}$ flops,
versus $O(kn\log^2(n)+nr^2)+a_{rn}$ involved in
Algorithm \ref{alg2}.

\begin{remark}\label{realph}
One can choose a positive $\alpha$ of Algorithm \ref{alg2a}  by applying 
heuristic methods or as follows:  map the two lines $\{x:~\Im x=\pm \alpha\}$
into the unit circle $C(0,1)$, extend these two maps to the
two maps of the  polynomial $p(x)$ into the polynomials 
$q_{\pm}(x)=p(x\pm  \alpha \sqrt {-1})$,
and apply the algorithm of Theorem \ref{thrrd} to these two polynomials.
\end{remark}


\subsection{Square Root Iteration (a Modified Modular Version)}\label{smbms2}


Next we describe another dual polynomial version of Algorithm \ref{alg2}.
It extend the square root iteration $y_{h+1}=\frac{1}{2}(y_h+1/y_h)$, $h=0,1,\dots$.
Compared to Algorithm \ref{alg2}, we first replace
all rational functions in the matrix $C_p$ by the same rational functions
in the variable $x$ and then reduce every function  modulo
the input polynomial $p(x)$.
The reduction does not affect the values of the functions at the roots
of $p(x)$, and so these values are precisely the eigenvalues of
the rational matrix functions involved in Algorithm \ref{alg2}.

\begin{algorithm}\label{alg3} {\rm Square root modular iteration modified for real root-finding.}

    \item{\textsc{Input and Output} as in Algorithm \ref{alg1}.}


    \item{\textsc{Computations:}}
         \begin{enumerate}
         \item
         Write $y_0=x$ and (cf. (\ref{eqmsn})) compute the polynomials
\begin{equation}\label{eqsqrt}
y_{h+1}=\frac{1}{2}(y_h-1/y_h)\mod p(x),~h=0,1,\dots.
\end{equation}

         \item
         Periodically, for selected integers $k$,
compute the polynomials  $t_k=y_k^2+1\mod p(x)$
and $g_k(x)={\rm agcd}(p,t_k)$.

     \item
If $\deg(g_k(x))=n-r=2s$,
compute the polynomial $v_k\approx p(x)/g_k(x)$
of degree $r$. Otherwise  continue the iteration of Stage 1.

     \item
     Apply one of the algorithms of  \cite{BT90}, \cite{BP98},
 and \cite{DJLZ97}
         (cf. Theorem \ref{thmlg}) to
     approximate the $r$ roots $y_1,\dots,y_r$ of the polynomial $v_k$.
 Output these approximations.
        \end{enumerate}
    \end{algorithm}

By virtue of our comments preceding this algorithm, the values of the polynomials
$t_k(x)$ at the roots of $p(x)$ are equal to the images
of the eigenvalues of the matrix $C_p$ in Algorithm \ref{alg2}.
Hence the values of the polynomials $t_k(x)$ at the nonreal roots 
of $p(x)$ 
converge to 0
as $k\rightarrow \infty$, whereas their values at
 the real roots of $p(x)$ stay far from 0. Therefore, for sufficiently large integers $k$,
 ${\rm agcd}(p,t_k)$ turn into the polynomial $\prod_{j=r+1}^n(x-x_j)$.
This implies correctness of the algorithm.
Its asymptotic computational cost is $O(kn\log^2(n))$ plus the cost of computing
${\rm agcd}(p,t_k)$ and
choosing the integer $k$ (see our next remark).

\begin{remark}\label{resmlr}
Compared to Algorithm \ref{alg2}, the latter algorithm
reduces real root-finding essentially
to the computation of agcd$(p,t_k)$,
but the complexity of this computation is not easy to estimate \cite{BB10}.
Moreover,  let us reveal serious problems of numerical
stability for this algorithm and for
the similar algorithms of \cite{C96} and \cite{BP96}.
Consider the case where $r=0$.
Then the polynomial $t(x)$ has degree at most $n-1$,
and its values at the $n$ nonreal roots of the polynomial
$p(x)$ are close to 0. This can only occur if $||t_k(x)||\approx 0$.
\end{remark}

\begin{remark}\label{refr1}
We can concurrently perform
Stages 1 of both Algorithms \ref{alg2} and \ref{alg3}.
The information about the numerical rank at Stage 3 of
Algorithm \ref{alg2} can be
 a guiding rule for the choice of the integer parameter $k$
and computing the polynomials $t_k$,
$g_k$ and $v_k$ of
Algorithm \ref{alg3}. Having the polynomial $v_k$
available, Algorithm \ref{alg3} produces the
approximations to the real roots more readily
than  Algorithm \ref{alg2} does this at its Stage 4.
\end{remark}


\subsection{Cayley Map and Root-squaring}\label{scmrs}


The following algorithm is somewhat similar to Algorithm
\ref{alg1}, but employs repeated squaring of the roots instead of
mapping them into their square roots.

\begin{algorithm}\label{alg4} {\rm Real root-finding with Cayley map and 
repeated root-squaring.}

\item{\textsc{Input and Output} as in Algorithm \ref{alg1},
except that we require that  $p(1)p(\sqrt {-1})\neq 0$.}

\item{\textsc{Computations:}}
         \begin{enumerate}
         \item
 Compute the polynomial $q(x)=(\sqrt{-1}~(x-1)^n\frac{x+1}{x-1})=\sum_{i=0}^nq_ix^i$.
(This is the Cayley map of Theorem \ref{thcl}. It moves the real axis, in particular the real roots of $p(x)$, onto the unit circle $C(0,1)$.)

 \item
 Write $q_0(x)=q(x)/q_n$, fix a sufficiently large integer
$k$, and apply the $k$ squaring steps of Theorem \ref{thdnd},
$q_{h+1}(x)=(-1)^nq_h(\sqrt{x})q_h(-\sqrt{x})$ for
$h=0,1,\dots,k-1$. (These steps keep the images of the real roots of
$p(x)$ on the circle $C(0,1)$ for all $k$, while sending the
images of every other root of $p(x)$ toward either the origin or
the infinity.)

 \item
For a sufficiently large integer $k$,  the polynomial $q_{k}(x)$
approximates the polynomial $x^su_k(x)$ where
$u_k(x)$ 
is a polynomial of degree $r$ 
whose all $r$ roots  lie on the unit circle $C(0,1)$.
 Extract an approximation to this polynomial  from the coefficients
of the polynomial  $q_{k}(x)$.

 \item
Compute the polynomial
$w_k(x)=u_k(\sqrt{-1}~\frac{x+1}{x-1})$.
(This Cayley map sends the images of 
the real roots of the polynomial $p(x)$ from the unit circle
$C(0,1)$ back to the real line.)

 \item
Apply one of the algorithms of  \cite{BT90}, \cite{BP98}, and \cite{DJLZ97}
 to approximate the $r$ real roots $z_1,\dots,z_r$ of the polynomial $w_k(x)$
         (cf. Theorem \ref{thmlg}).

 \item
Apply the Cayley map $w_j^{(k)}=(z_j +\sqrt{-1})/(z_j-\sqrt{-1})$ for $j=1,\dots,r$
to extend Stage 5 to approximating the $r$ roots $x_1^{(k)},\dots,x_r^{(k)}$
of the polynomials $u_k(x)$ and $y_k(x)=x^su_k(x)$
lying on the unit circle $C(0,1)$.

 \item
 Apply the descending process (similar to the ones of \cite{P95}, \cite{P02},
and of our Algorithm \ref{alg1}) to approximate the $r$ roots $x_1^{(h)},\dots,x_r^{(h)}$
of the polynomials $q_{h}(x)$ lying on the unit circle $C(0,1)$ for $h=k-1,\dots,0$.

 \item
 Approximate the $r$ real roots
$x_j=\sqrt{-1}(x_j^{(0)} +1)/(x_j^{(0)}-1)$, $j=1,\dots,r$,
of the polynomials $p(x)$.
 \end{enumerate}
\end{algorithm}

Our analysis of Algorithm \ref{alg1} (including its complexity estimates
and the comments and recipes in Remarks \ref{redgn}--\ref{renrr})
can be extended to Algorithm \ref{alg4}.


\subsection{A Tentative Approach to Real Root-finding by Means of 
Root-radii Approximation}\label{srra}


\begin{algorithm}\label{algrrnt} ({\em Real root-finding by means of root radii approximation}.)
\item{\textsc{Input and Output} as in Algorithm \ref{alg1}.}

\item{\textsc{Computations:}}
         \begin{enumerate}
         \item
 Compute 
approximations $\tilde r_1,\dots,\tilde r_n$
to the  root radii of a polynomial $p(x)$ of (\ref{eqpoly}) (see Theorem \ref{thrrd}). (This defines $2n$ candidates points
$\pm \tilde r_1,\dots,\pm \tilde r_n$ for the approximation of the $r$
real roots $x_1,\dots,x_r$.)

\item
At all of these $2n$ points, apply one of the proximity 
tests of Section \ref{srrd},
 to select $r$ approximations to the $r$ real roots of 
the polynomial $p(x)$.

\item
 Apply Newton's iteration
$x^{(h+1)}=x^{(h)}-p(x^{(h)})/p'(x^{(h)}),~ h=0,1,\dots$,
concurrently at these $r$ points,
expecting to refine quickly the approximations to 
the isolated simple real roots.
  \end{enumerate}
\end{algorithm}


\section{Real Eigen-solving for a Real Nonsymmetric Matrix}\label{sreig}


Suppose we are seeking the real eigenvalues of
a real nonsymmetric $n\times n$ matrix $M$.
We can substitute this matrix
for the input matrix $C_p$ of 
Algorithm \ref{alg2} or \ref{alg2a} and apply 
the algorithm with no further changes.
The overall arithmetic complexity would grow to $O(kn^3)$ flops, 
but may still be competitive
if the integer $k$ is small, that is, if
the algorithm converges fast for the input matrix $M$. 

Seeking acceleration,
one can first define a similarity transformation 
of the matrix $M$ into a rank structured matrix
whose all
subdiagonal blocks have rank at most 1 \cite{VVM}, \cite{EGH13}.
Then one would only need $O(n^2)$ flops to perform
the first iteration (\ref{eqmsn}), but each 
new iteration (\ref{eqmsn})
would double the upper bound on the maximal rank of the
subdiagonal blocks and thus would increase the estimated
complexity of the next iteration accordingly. So the overall arithmetic
cost would still be of order $kn^3$ flops, unless the integer $k$ is small. 
One is challenged to devise a similarity transformation of a matrix
that would decrease the  maximal rank of its subdiagonal block,
say, from 2 to 1, by using quadratic arithmetic time. 
This would decrease the overall arithmetic cost bound to $O(kn^2)$.

Now consider extension of  Algorithm \ref{alg4} to real eigen-solving. 
We must avoid using high powers of the input and auxiliary matrices because
these powers tend to have numerical rank 1. The following 
algorithm, however, involves 
such powers implicitly, when it computes the auxiliary matrix
$P^m-P^{-m}$
as the product $\prod_{i=0}^{m-1}(P-\omega_m^iP^{-1})$
where $P=(M+\sqrt{-1}~I)(M-\sqrt{-1}~I)^{-1}$, $m$ denotes a fixed
reasonably large integer,
 and
$\omega_m=\exp(2\pi\sqrt {-1}/m)$ is a primitive $m$th root of unity.

\begin{algorithm}\label{alg6} {\rm Real eigen-solving by means of factorization.}

\item{\textsc{Input:}} a real $n\times n$ matrix $M$
having $r$ real eigenvalues  and $s=(n-r)/2$ pairs of 
nonreal complex conjugate eigenvalues, neither of them
is equal to $\sqrt{-1}$.
    \item{\textsc{Output:}} approximations to the real eigenvalues $x_1,\dots,x_r$ of the matrix $M$.

\item{\textsc{Computations:}}
         \begin{enumerate}
         \item
 Compute the matrix $P=(M+\sqrt{-1}~I)(M-\sqrt{-1}~I)^{-1}$.
(This is the matrix version of a Cayley map of Theorem \ref{thcl}.
It moves the real and only the real eigenvalues of the matrix $M$ into 
the eigenvalues of the matrix $P$ lying on the unit circle $C(0,1)$.)

 \item
 Fix a sufficiently large integer
$m$ and compute the matrix $Y=(P^m-P^{-m})^{-1}$
in the following factorized form $\prod_{i=0}^{m-1}(P-\omega_m^iP^{-1})^{-1}$
where $\omega_m=\exp(2\pi\sqrt {-1}/m)$.
 (For any integer $m$ the images of all real eigenvalues of the matrix $M$
have absolute values at least 1/2,
 whereas the 
images of all nonreal eigenvalues of that matrix 
converge to 0 as $m\rightarrow \infty$.)

 \item
Complete the computations as at Stages 3 and 4 
of Algorithm \ref{alg2}.

 \end{enumerate}
\end{algorithm}

The arithmetic complexity of the algorithm is $O(mn^3)$ flops for general matrix $M$,
but decreases to $O(mn^2)$ if  $M$ is a Hessenberg matrix or if the rank of
all its subdiagonal blocks is bounded by a constant. For $M=C_p$
the complexity decreases  to $O(mn)$,
which makes the algorithm attractive for real polynomial root-finding,
as long as it converges for a reasonably small integers $m$.

\begin{remark}\label{rescl}({\rm Scaling and the simplification of the factorizations.})
One 
can apply the algorithm to a scaled matrix $\theta M/||M||$
for a fixed matrix norm $||\cdot||$
and a fixed scalar $\theta$, $0<\theta<1$, say, for $\theta=0.5$. In this case 
the inversion at Stage 1 is applied to a diagonally dominant
matrix. Towards more radical simplification of the algorithm,
one can avoid computing and inverting the matrix $P$ and can instead 
compute the matrix $Y$ in one of the following two equivalent factorized forms,
$$Y=\prod_{i=0}^{m-1}((M^2+I)~F_i(M)^{-1}G_i(M)^{-1})=
\prod_{i=0}^{m-1}(\alpha_iF_i(M)^{-1}+\beta_i G_i(M)^{-1})$$
for
$$F_i(M)=M+\sqrt{-1}~I+\omega_{2m}^i(M-\sqrt{-1}~I)=
(1+\omega_{2m}^i)M+\sqrt{-1}(1-\omega_{2k}^i)I,$$
$$G_i(M)=M+\sqrt{-1}~I-\omega_{2m}^i(M-\sqrt{-1}~I)=
(1-\omega_{2m}^i)M+\sqrt{-1}(1+\omega_{2m}^i)I,$$
some complex scalars $\alpha_i$ and $\beta_i$, and 
$i=0,\dots,m-1$. Then again, one can 
apply the algorithm to a scaled matrix $\gamma M$
for an appropriate scalar $\gamma$
to simplify the solution of linear systems of 
equations with the matrices $F_i(M)$ and $G_i(M)$.
\end{remark}

\begin{remark}\label{redbl}
One can adapt the integer $m$ by doubling it  to produce the desired eigenvalues
if the computations
show that the current integer $m$
 is not large enough.
The previously computed  matrices $F_i(M)$ and $G_i(M)$ can be reused.
\end{remark}


\section{Numerical Tests}\label{snm}


Three series of numerical tests have been performed in the Graduate
Center of the City City University of New York by Ivan Retamoso
and Liang Zhao. In all three series they tested Algorithm \ref{alg2},
and the
results of the  test are quite encouraging.

In the first series of tests, Algorithm \ref{alg2} has been applied to one of the Mignotte
benchmark polynomials, namely to $p(x)=x^n + (100x-1)^3$. It is known that this
polynomial has three ill conditioned roots clustered about $0.01$
and has $n-3$ well conditioned  roots. In the tests, Algorithm \ref{alg2} has output
the roots within the error less than $10^{-6}$ by using 9 iterations for $n=32$ and  $n=64$
and by using 11 iterations for $n=128$ and $n=256$.

In the second series of tests, polynomials $p(x)$ of degree
$n=50,100,150,200$, and $250$ have been generated as the products
$p(x)=f_{1}(x)f_{2}(x)$. Here $f_{1}(x)$ was the $r$th degree  Chebyshev polynomial 
(having $r$ real roots) for $r=8,12,16$, 
and $ f_{2}(x)=\sum^{n-r}_{i=0}a_{i}x^{i}$, $a_{j}$ being i.i.d. standard Gaussian random variables,
for  $j=0,\dots,n-r$.
Algorithm  \ref{alg2} (performed with double precision) was
applied to 100 such polynomials $p(x)$ for each pair
of $n$ and $r$.
 Table \ref{tabhank} displays the output data,
 namely, the average values and standard deviation of 
the numbers of iterations and of the maximum difference between the output values
of the roots and their values produced by MATLAB root-finding
function "roots()".

In the third series of tests,
Algorithm \ref{alg2} 
 approximated the real eigenvalues of a  random real symmetric matrix
 $A=U^{T}\Sigma U$, where $U$ was an orthogonal $n\times n$ standard 
Gaussian random matrix, $\Sigma=\diag(x_{1},\dots,x_{r},y_{1},\dots, y_{n-r})$, 
and $x_{1},\dots,x_{r}$ (resp. $y_{1},\dots, y_{n-r}$) were $r$ i.i.d. standard Gaussian real (resp. non-real) random variables.
 Table \ref{tabhank2} displays the mean and standard deviation of
the number of iterations and the error bounds in these tests
for $n=50,100,150,200,250$ and $r=8,12,16$.

\begin{table}[h]
\caption{Number of Iterations and Error Bounds for 
Algorithm \ref{alg2} on Random Polynomials}
\label{tabhank}
  \begin{center}

    \begin{tabular}{| c | c | c | c | c | c | }

\hline

\bf{n} & \bf{r} & \bf{Iter-mean} & \bf{Iter-std} & \bf{Bound-mean} & \bf{Bound-std} \\ \hline

$50$ & $8$ & $7.44$ & $1.12$ & $4.18\times 10^{-6}$ & $1.11\times 10^{-5}$\\ \hline
$100$ & $8$ & $8.76$ & $1.30$ & $5.90\times 10^{-6}$ & $1.47\times 10^{-5}$\\ \hline
$150$ & $8$ & $9.12$ & $0.88$ & $2.61\times 10^{-5}$ & $1.03\times 10^{-4}$\\ \hline
$200$ & $8$ & $9.64$ & $0.86$ & $1.48\times 10^{-6}$ & $5.93\times 10^{-6}$\\ \hline
$250$ & $8$ & $9.96$ & $0.73$ & $1.09\times 10^{-7}$ & $5.23\times 10^{-5}$\\ \hline

$50$ & $12$ & $7.16$ & $0.85$ & $3.45\times 10^{-4}$ & $9.20\times 10^{-4}$\\ \hline
$100$ & $12$ & $8.64$ & $1.15$ & $1.34\times 10^{-5}$ & $2.67\times 10^{-5}$\\ \hline
$150$ & $12$ & $9.12$ & $2.39$ & $3.38\times 10^{-4}$ & $1.08\times 10^{-3}$\\ \hline
$200$ & $12$ & $9.76$ & $2.52$ & $6.89\times 10^{-6}$ & $1.75\times 10^{-5}$\\ \hline
$250$ & $12$ & $10.04$ & $1.17$ & $1.89\times 10^{-5}$ & $4.04\times 10^{-5}$\\ \hline

$50$ & $16$ & $7.28$ & $5.06$ & $3.67\times 10^{-3}$ & $7.62\times 10^{-3}$\\ \hline
$100$ & $16$ & $10.20$ & $5.82$ & $1.44\times 10^{-3}$ & $4.51\times 10^{-3}$\\ \hline
$150$ & $16$ & $15.24$ & $6.33$ & $1.25\times 10^{-3}$ & $4.90\times 10^{-3}$\\ \hline
$200$ & $16$ & $13.36$ & $5.38$ & $1.07\times 10^{-3}$ & $4.72\times 10^{-3}$\\ \hline
$250$ & $16$ & $13.46$ & $6.23$ & $1.16\times 10^{-4}$ & $2.45\times 10^{-4}$\\ \hline

\end{tabular}

  \end{center}

\end{table}

\begin{table}[h]
\caption{Number of Iterations and Error Bounds for Algorithm 
 \ref{alg2} on Random Matrices}
\label{tabhank2}
  \begin{center}

    \begin{tabular}{| c | c | c | c | c | c | }

\hline

\bf{n} & \bf{r} & \bf{Iter-mean} & \bf{Iter-std} & \bf{Bound-mean} & \bf{Bound-std} \\ \hline

$50$ & $8$ & $10.02$ & $1.83$ & $5.51\times 10^{-11}$ & $1.65\times 10^{-10}$\\ \hline
$100$ & $8$ & $10.81$ & $2.04$ & $1.71\times 10^{-12}$ & $5.24\times 10^{-12}$\\ \hline
$150$ & $8$ & $14.02$ & $2.45$ & $1.31\times 10^{-13}$ & $3.96\times 10^{-13}$\\ \hline
$200$ & $8$ & $12.07$ & $0.94$ & $2.12\times 10^{-11}$ & $6.70\times 10^{-11}$\\ \hline
$250$ & $8$ & $13.59$ & $1.27$ & $2.75\times 10^{-10}$ & $8.14\times 10^{-10}$\\ \hline

$50$ & $12$ & $10.46$ & $1.26$ & $1.02\times 10^{-12}$ & $2.61\times 10^{-12}$\\ \hline
$100$ & $12$ & $10.60$ & $1.51$ & $1.79\times 10^{-10}$ & $3.66\times 10^{-10}$\\ \hline
$150$ & $12$ & $11.25$ & $1.32$ & $5.69\times 10^{-8}$ & $1.80\times 10^{-7}$\\ \hline
$200$ & $12$ & $12.36$ & $1.89$ & $7.91\times 10^{-10}$ & $2.50\times 10^{-9}$\\ \hline
$250$ & $12$ & $11.72$ & $1.49$ & $2.53\times 10^{-12}$ & $3.84\times 10^{-12}$\\ \hline

$50$ & $16$ & $10.10$ & $1.45$ & $1.86\times 10^{-9}$ & $5.77\times 10^{-9}$\\ \hline
$100$ & $16$ & $11.39$ & $1.70$ & $1.37\times 10^{-10}$ & $2.39\times 10^{-10}$\\ \hline
$150$ & $16$ & $11.62$ & $1.78$ & $1.49\times 10^{-11}$ & $4.580\times 10^{-11}$\\ \hline
$200$ & $16$ & $11.88$ & $1.32$ & $1.04\times 10^{-12}$ & $2.09\times 10^{-12}$\\ \hline
$250$ & $16$ & $12.54$ & $1.51$ & $3.41\times 10^{-11}$ & $1.08\times 10^{-10}$\\ \hline

\end{tabular}

  \end{center}

\end{table}







\end{document}